\documentclass[ a4paper,10pt]{amsart}
\usepackage[foot]{amsaddr}

\usepackage[T1]{fontenc}

\usepackage[utf8]{inputenc}

\usepackage{lmodern}

\usepackage[english]{babel}

\usepackage{amsmath}

\usepackage{amssymb}

\usepackage{mathtools}

\usepackage{mathrsfs}

\usepackage{xypic}

\usepackage{tikz-cd}

\usepackage{bookmark}

\usepackage{hyperref} 

\usepackage[alphabetic]{amsrefs}

\usepackage{enumitem}
\usepackage{graphicx}
\usepackage{tikz}
\usepackage{tikz-cd}
\usetikzlibrary{matrix,arrows,decorations.pathmorphing}
\usepackage[all,cmtip]{xy}

\newtheorem*{thm-no-num}{Theorem}
\newtheorem*{cor-no-num}{Corollary}
\newtheorem*{df-no-num}{Definition}
\newtheorem{thm}{Theorem} [section]
\newtheorem{prop}[thm]{Proposition} 
 
\newtheorem{cor}[thm]{Corollary} 
\theoremstyle{remark}
\newtheorem{rmk}[thm]{Remark}
\newtheorem{ex}[thm]{Example}
\theoremstyle{definition} 
\newtheorem {df}[thm]{Definition}

\newcommand{\bA}{\mathbb{A}}
\newcommand{\Bcal}{\mathcal{B}}
\newcommand{\Ccal}{\mathcal{C}}
\newcommand{\Dcal}{\mathcal{D}}
\newcommand{\Fcal}{\mathcal{F}}
\newcommand{\Gm}{\mathbb{G}_m}
\newcommand{\Hcal}{\mathcal{H}}
\newcommand{\Lcal}{\mathcal{L}}
\newcommand{\Mcal}{\mathcal{M}}
\newcommand{\Ncal}{\mathcal{N}}
\newcommand{\Ocal}{\mathcal{O}}
\newcommand{\Pcal}{\mathcal{P}}
\newcommand{\sbold}{\mathbf{s}}
\newcommand{\Ucal}{\mathcal{U}}
\newcommand{\Vcal}{\mathcal{V}}
\newcommand{\Xcal}{\mathcal{X}}
\newcommand{\ZZ}{\mathbb{Z}}

\newcommand{\M}{{\rm M}}
\newcommand{\K}{{\rm K}}
\renewcommand{\H}{{\rm H}}

\DeclareMathOperator{\Het}{H^{\bullet}}
\DeclareMathOperator{\id}{id}
\DeclareMathOperator{\Inv}{Inv^{\bullet}}
\DeclareMathOperator{\Pic}{Pic}
\DeclareMathOperator{\PGL}{PGL}
\DeclareMathOperator{\pr}{pr}
\DeclareMathOperator{\Spec}{Spec}
\DeclareMathOperator{\USpec}{\underline{Spec}}
\DeclareMathOperator{\Sym}{Sym}

\begin{document}
\title[Invariants of root stacks and admissible double coverings]{Cohomological invariants of root stacks and admissible double coverings}
\author[A. Di Lorenzo]{Andrea Di Lorenzo}
	\address{Aarhus University, Ny Munkegade 118, DK-8000 Aarhus C, Denmark}
	\email{andrea.dilorenzo@math.au.dk}
\author[R. Pirisi]{Roberto Pirisi}
	\address{KTH Royal Institute of Technology, Lindstedtsvägen 25, 10044 Stockholm, Sweden}
	\email{pirisi@kth.se}	
\date{\today}
\begin{abstract}
We give a formula for the cohomological invariants of a root stack, which we apply to compute the cohomological invariants and the Brauer group of the stack of admissible double coverings. 

2010 MSC classification: 14F22, 14H10
\end{abstract}
\maketitle

\section*{Introduction}\label{sec:intro}



Cohomological invariants of algebraic groups are a classical topic, whose first examples date back to the beginning of the twentieth century \cite{Wit}. 
The modern definition was developed by Serre in the late '90s, and the general theory can be found in Garibaldi, Merkurjev and Serre's book \cite{GMS}.
From a more geometric point of view, they can be seen as invariants of the classifying stack ${\rm B}G$ rather than the group $G$. Following this idea, in \cite{PirAlgStack}, the second author defined an extension of the notion of cohomological invariants to general algebraic stacks.

Let $\Xcal$ be an algebraic stack  over a field $k$ and $\M$ a cycle module over $k$, as defined in \cite{Rost}. A cohomological invariant for $\Xcal$ with coefficients in $\M$ is a functorial way to associate to any point $\Spec(K) \rightarrow \Xcal$ an element of $\M^{\bullet}(K)$, satisfying a certain continuity condition. Originally, in \cite{GMS}, cohomological invariants were defined as taking values in the cycle module ${\rm H}_D$, where $D$ is an $\ell$-torsion Galois module for some positive number $\ell$ not divisible by ${\rm char}(k)$, and ${\rm H}_D^{\bullet}(F):=\oplus_i {\rm H}^i_{\rm Gal}(F,D(i))$. The theory for general cycle modules was recently developed by Gille and Hirsch \cite{GilHir}, and adapted to the setting of cohomological invariants of algebraic stacks by th authors in \cite{DilPir2}.

Our definition retrieves the classical theory of cohomological invariants of algebraic groups when $\Xcal={\rm B}G$, and if $\Xcal=X$ is a scheme its cohomological invariants with coefficients in ${\rm H}_D$ are equal to its unramified cohomology ${\rm H}^{\bullet}_{\rm nr}(X,D)$.
Moreover, in \cite{DilPir2}, the authors proved that the cohomological invariants compute the Brauer group, and used them to compute the Brauer groups of the moduli stacks $\Hcal_g$ of hyperelliptic curves of genus $g$.

In \cite{DilPir}*{Appendix A}, the authors showed that the cohomological invariants (and consequently, by \cite{DilPir2}*{Rmk. 6.11}, the Brauer group) of the standard compactification $\overline{\Hcal_g}$ of $\Hcal_g$ are trivial, while claiming this to not occur for the compactification $\widetilde{\Hcal}_g$ obtained via the theory of admissible double coverings. In this paper we prove the aforementioned claim:

\begin{thm-no-num}[Thms \ref{thm:inv even}, \ref{thm:inv odd}]
Let $\M$ be an $\ell$-torsion cycle module. Then:
\begin{itemize}
    \item If $\ell$ is odd, we have $\Inv(\widetilde{\Hcal}_g,\M)=\M^{\bullet}(k)$.
    \item If $\ell=2^m$, then we have an exact sequence
    $$0 \rightarrow \Inv({\rm BS}_{2g+2},\M) \rightarrow \Inv(\widetilde{\Hcal}_g,\M) \rightarrow \M^{\bullet}(k)_2 ,$$
\end{itemize}
where the last map lowers degree by $g+2$. Moreover, if $g$ is even, the last map is surjective and split.
\end{thm-no-num}

Thanks to the result above we are able compute the prime-to-$c$ part of the Brauer group of $\widetilde{\Hcal}_g$, which coincides with the whole Brauer group of the stack when the base field has characteristic zero:

\begin{cor-no-num}[Thm \ref{thm:Brauer}]
We have 
$$^c{\rm Br}(\widetilde{\Hcal}_g)=^c\!{\rm Br}(k) \oplus {\rm H}^{1}(k,\mathbb{Z}/2\mathbb{Z}) \oplus \mathbb{Z}/2\mathbb{Z}. $$
\end{cor-no-num}

If one considers the analogy with unramified cohomology, this appears to be in stark contrast to what happens for proper smooth varieties, where unramified cohomology is a birational invariant. The stacks $\overline{\Hcal_g}$ and $\widetilde{\Hcal}_g$ are smooth, proper (in fact, projective, according to Kresch's definition of a projective Deligne-Mumford stack \cite{KreProj}) and birational to each other, but their cohomological invariants are very different. The reason why this happens is related to the so called root stack construction.

\subsection*{Weak factorization, root stacks and cohomological invariants}
Given smooth projective varieties $X,Y$ over a base field $k$ of characteristic zero, the weak factorization theorem states that a birational map $X \dashrightarrow Y$, which is an isomorphism on an open subset $U$, always factors as a composition
$$X=X_0 \overset{f_0}{\dashrightarrow}X_1 \overset{f_1}{\dashrightarrow} \ldots \overset{f_{n-1}}{\dashrightarrow} X_n = Y$$
where for each $i$ either $f_i$ or $f_i^{-1}$ is a blow-up in a smooth center and each $f_i$ is an isomorphism on $U$.

Recently, Harper \cite{Har} proved a similar result for Deligne-Mumford stacks over a field of characteristic zero. Given a proper birational map $\mathcal{X} \rightarrow \mathcal{Y}$, we always have a factorization
$$\mathcal{X}=\mathcal{X}_0 \overset{f_0}{\dashrightarrow}\mathcal{X}_1 \overset{f_1}{\dashrightarrow} \ldots \overset{f_{n-1}}{\dashrightarrow} \mathcal{X}_n = \mathcal{Y}$$
where for each $i$ either $f_i$ or $f_i^{-1}$ is a \emph{stacky blow-up}, that is either a regular blow up in a smooth center or a root stack along a smooth divisor. Given a Cartier divisor $V$, the order $r$ root stack along $V$, denoted $\mathcal{X}_{(V,r)} \rightarrow \mathcal{X}$, is a canonical way to add a $\mu_r$ stabilizer along the points of $V$. 

Unramified cohomology and cohomological invariants are left unchanged when taking a blow-up, but can change drastically when taking a root stack. If one wants to be able to compare the cohomological invariants of birational proper Deligne-Mumford stacks, it is critical to be able to understand what happens when taking a root stack. This is described in the following Theorem:

\begin{thm-no-num}[Thm. \ref{thm:root Nis}]
Let $\Xcal$ be a smooth algebraic stack over a field $k$, and let $r\geq2$ be an integer that is not divisible by ${\rm char}(k)$. Let $V \subset \Xcal$ be a divisor. 

Then $\Inv(\Xcal_{(V,r)},\M)$ is equal to the subset of $\Inv(\Xcal\setminus V,\M)$ of invariants $\alpha$ such that the ramification of $\alpha$ at $V$ is of $r$-torsion.
\end{thm-no-num}

The stack $\widetilde{\Hcal}_g$ is, up to codimension $2$ closed substacks, a root stack over $\overline{\Hcal}_g$. More precisely, after excising some loci, it can be obtained from an open substack of $\overline{\Hcal}_g$ by adding a $\mu_2$ stabilizer along the divisor of irreducible singular curves.

As cohomological invariants are unchanged when removing a closed subset of codimension $\geq 2$, we can (and we do) apply the Theorem above to compute the cohomological invariants of $\widetilde{\Hcal}_g$, leveraging our previous knowledge of the cohomological invariants of $\Hcal_g$ and $\overline{\Hcal}_g$.

\subsection*{Cohomological invariants and Chow groups with coefficients}

We briefly recall some properties of cohomological invariants which will be used in the paper. For a much more detailed introduction to cohomological invariants and Chow groups with coefficients, the reader can refer to \cite{DilPir2}*{Sec. 1}.

Given a field $F$, let ${\rm T}^{\bullet}(F)$ be the tensor algebra $\oplus_{i\geq 0}(F^*)^{\otimes_{\ZZ} i}$. Milnor's ${\rm K}$-theory $\K_{\rm M}$ is the functor from field extensions of $k$ to graded-commutative rings which sends a field $F$ to 
\[
\K_{\M}^{\bullet}(F) = {\rm T}^{\bullet}(F)/\lbrace a\otimes b \mid a,b \in F^*, \, a + b =1 \rbrace.
\]

Milnor's $\K$-theory is the basic, and most important, cycle module. A general cycle module $\M$ is a functor from fields over $k$ to abelian groups which is functorially a $\K_\M$-module, and shares the main properties of Milnor's $\K$-theory. Given a torsion Galois module $D$ over $k$, the functor 
\[
\H_D^{\bullet}(F)=\oplus_{i\geq 0}\H^i_{\rm Gal}(F,D(i))
\]
is a cycle module. In particular if we take $D=\ZZ/\ell\ZZ$ by Voevodsky's norm-residue isomorphism \cite{Voe}*{Thm. 6.1} we have
\[
\H_{\ZZ/\ell \ZZ}=\K_M/(\ell)\stackrel{{\rm def}}{=} \K_\ell.
\]
If $\M$ is of $\ell$-torsion, it is a $\K_\ell$-module.

Cohomological invariants with coefficients in a cycle module $\M$ form a graded abelian group. When $\M$ is given by $\K_{\M}$ or $\K_{\ell}$, they form a graded-commutative ring. In general, $\Inv(\Xcal,\M)$ is a $\Inv(\Xcal,\K_{\M})$-module, and if $\M$ is of $\ell$-torsion it is a $\Inv(\Xcal,\K_{\ell})$-module.

The functor $\Inv(-,\M)$ forms a sheaf with respect to the \emph{smooth-Nisnevich} tolopogy, where the coverings are smooth representable maps $\mathcal{Y} \to \Xcal$ such that any point $\Spec(K) \to \Xcal$ has a lifting $\Spec(K) \to \mathcal{Y}$. Moreover, cohomological invariants are unchanged when passing to an affine bundle or an open subset whose complement has codimension $\geq 2$, and in general restriction to a dense open subset is injective.

Our main tool for computatations will be Rost's Chow groups with coefficients \cite{Rost}. These are an extension of ordinary Chow groups: instead of having $\ZZ$-linear sums of cycles, they are generated by sums in the form $\sum_i \alpha_i P_i$, where $\alpha_i$ belongs to $\M^{\bullet}(k(P))$. The $i$-codimensional Chow group of $X$ with coefficients in $\M$ is denoted by $A^{i}(X,\M)$.

Chow groups with coefficients share the same properties of ordinary Chow groups, and moreover given a closed subset $V$ of $X$ of pure codimension $d$ they have a long exact sequence
\[ \ldots \to A^i(X,D) \to A^i(X\setminus V,D) \to A^{i+1-d}(V,D) \to A^{i+1}(X,D) \to \ldots
\]
When $X$ is a smooth scheme, we have an isomorphsim $A^0(X,\M) \simeq \Inv(X,\M)$. 

More generally, if $X$ is a smooth scheme being acted upon by a smooth affine algebraic group $G$ we can use the Edidin-Graham-Totaro approximation process \cite{EG} to define equivariant Chow groups with coefficients $A^i_{G}(X,\M)$. This was first done by Guillot \cite{Guil}. Again, we have
\[ \Inv(\left[X/G\right],\M) \simeq A^{0}_{G}(X,\M).\]

\subsection*{Outline of the paper}
In Section \ref{sec:root} we discuss the root stack construction in greater detail. In particular, we give a presentation of the root stack $X_{(V,r)}$ as a quotient (Proposition \ref{prop:root as quotient}), which we use to deduce our formula for cohomological invariants (Theorem \ref{thm:root stack}).

In Section \ref{sec:stable} we recall some basic facts on the stack of stable hyperelliptic curves and we provide a new description of an open substack of $\overline{\Hcal}_g$, which will be useful for our purposes.

In Section \ref{sec:admissible} we describe the theory of admissible double coverings and its connection to hyperelliptic curves.

Finally, in Section \ref{sec:computations}, we compute the cohomological invariants of $\widetilde{\Hcal}_g$ (Theorem \ref{thm:inv even} for the even genus case and Theorem \ref{thm:inv odd} for the odd genus case), as well as its Brauer group (Theorem \ref{thm:Brauer}).

\subsection*{Notation and conventions}
Throughout the paper we will fix a base field $k$ of characteristic $c\neq 2$. All stacks and schemes are assumed to be of finite type over $k$ unless otherwise specified. By $\ell$ or $r$ we will always mean a positive integer that is not divisible by the characteristic of $k$. If $A$ is an abelian group, we denote $A_{\ell}$ the $\ell$-torsion subgroup. If moreover $A$ is torsion, we denote $^c\!A$ the subgroup of elements whose order is not divisible by $c$.

\section{Cohomological invariants of root stacks}\label{sec:root}
Fix an integer $r\geq 1$ and let $X$ be an integral scheme over some field $k$ whose characteristic does not divide $r$. Choose an invertible sheaf $\Lcal$ over $X$ and define $X_{(\Lcal,r)}$ as the root stack of order $r$ of $X$ along the invertible sheaf $\Lcal$: this is the stack over the big \'{e}tale site of $X$ whose objects are triples $(S\to X,\Mcal,\phi)$, where $\Mcal$ is an invertible sheaf over $S$ together with an isomorphism $\phi:\Lcal_S\simeq\Mcal^{\otimes r}$. A morphism $(S\to X,\Mcal,\phi)\to (S\to X,\Mcal',\phi')$ over $S$ is given by an isomorphism of invertible sheaves $\psi:\Mcal\simeq\Mcal'$ such that $\phi'\circ\psi^{\otimes r}=\phi$.

The root stack $X_{(\Lcal,r)}$ has a simple presentation as quotient stack: let $L$ be the line bundle associated to the invertible sheaf $\Lcal$, i.e. $L:=\USpec_{\Ocal_X}(\Sym \Lcal^{\vee})$, and denote $L^*$ the complement of the zero section in $L$.

The natural $\Gm$-action on $L$ turns $L^*\to X$ into a $\Gm$-torsor. For every $r\geq 1$ we have a well defined action  of $\Gm^{\otimes r}$ over $L$, via the obvious multiplication homomorphism $\Gm^{\otimes r}\to \Gm$.
\begin{prop}\label{prop:root as quotient simple}
    We have $X_{(\Lcal,r)}\simeq \left[ L^*/\Gm^{\otimes r} \right]$.
\end{prop}
\begin{proof}
    A $\Gm^{\otimes r}$-torsor $P^*$ over $S$ is the complement in a line bundle $P\to S$ of the image of the zero section, where $P$ is of the form $P=\USpec_{\Ocal_S}(\Sym (\Mcal^{\otimes r})^{\vee})$. 
    
    Given a morphism $S\to X$, a $\Gm^{\otimes r}$-equivariant morphism $P^*\to L^*_S$ is necessarily an isomorphism, which induces an isomorphism of invertible sheaves $\Lcal\simeq \Mcal^{\otimes r}$.
    
    By definition, the groupoid of maps $S\to \left[ L^*/\Gm^{\otimes r}\right]$ is equivalent to the groupoid of $\Gm^{\otimes r}$-torsors $P^*\to S$ endowed with an equivariant morphism to $L^*$: by what we have just observed, the latter is equivalent to the groupoid of maps $S\to X$ together with an invertible sheaf $\Mcal$ over $S$ and an isomorphism $\Lcal_S\simeq\Mcal^{\otimes r}$. This shows that $X_{(\Lcal, r)}(S)\simeq \left[ L^*/\Gm^{\otimes r} \right](S)$ for every $X$-scheme $S$ and thus concludes the proof.
\end{proof}

Consider a non-zero global section $\sigma$ of $\Lcal$, and let $V$ be the divisor where the section $\sigma$ vanishes. The stack $X_{(V,r)}$ is the root stack of order $r$ of $X$ along $V$: its objects are quadruples $(S\to X,\Mcal,\tau,\phi)$ where $\Mcal$ is an invertible sheaf over $S$ and $\tau$ is a global section of $\Mcal$, together with an isomorphism $\phi:\Lcal_S\simeq\Mcal^{\otimes r}$ such that $\phi(\sigma)=\tau^{\otimes r}$.

The root stack $X_{(V,r)}$ also admits a simple description as a quotient stack: regard $\sigma$ as a morphism from $L$ to $\bA^1_X$, and let $(\cdot)^r:\bA^1_X\to\bA^1_X$ be the $r^{\rm th}$-power morphism, so that we can form the cartesian diagram
\[
\xymatrix{Y^* \ar[r] \ar[d] & \bA^1_X \ar[d]^{(\cdot)^r} \\
L^* \ar[r]^{\sigma} & \bA^1_X }
\]

Roughly, the scheme $Y$ can also be thought as the closed subscheme of $L^*\times \bA^1$ locally described by the equation $u^r=\sigma(x,t)$, where $x$ and $t$ are local coordinates for $L^*$ and $u$ is a coordinate for $\bA^1$.

We can define an action of $\Gm$ on $L^*$ and $\bA^1_X$ so that the morphisms $\sigma$ and $(\cdot)^r$ become equivariant: namely, we let $\Gm$ acts on both $L^*$ and the $\bA^1_X$ on the bottom right corner via its $r^{\rm th}$-power, and on the $\bA^1_X$ in the top right corner via the standard action. In this way $Y^*$ inherits a $\Gm$-action too, and we get a cartesian diagram of quotient stacks as follows:
\[
\xymatrix{\left[Y^*/\Gm\right] \ar[r] \ar[d] & \left[\bA^1_X/\Gm\right] \ar[d]\\
\left[ L^*/\Gm^{\otimes r} \right] \ar[r] & \left[\bA^1_X/\Gm^{\otimes r}\right]}
\]
\begin{prop}\label{prop:root as quotient}
    We have $X_{(V,r)}\simeq [Y^*/\Gm]$.
\end{prop}
\begin{proof}
    By Proposition \ref{prop:root as quotient simple} the groupoid $\left[ L^*/\Gm^{\otimes r} \right](S)$ is equivalent to the groupoid of objects $(S\to X, \Mcal, \Lcal_S\simeq \Mcal^{\otimes r})$.
    
    As for $\left[ \bA^1_X/\Gm \right]$, this stack is isomorphic to $X\times \left[\bA^1/\Gm\right]$. It is well known that the groupoid of maps from $S$ to $[\bA^1/\Gm]$ is equivalent to the groupoid of invertible sheaves $\Ncal$ over $S$ together with a global section $\tau$.
    
    The morphism $[\bA^1_X/\Gm]\to [\bA^1_X/\Gm^{\otimes r}]$ sends an object $(S\to X,\Ncal,\tau)$ to $(S\to X,\Ncal,\tau^{\otimes r})$. The morphism $[L^*/\Gm^{\otimes r}]\to [\bA^1_X/\Gm^{\otimes r}]$ sends an object $(S\to X,\Mcal,\phi:\Lcal\simeq\Mcal^{\otimes r})$ to the object $(S\to X,\Mcal,\phi\circ\sigma_S)$, where $\sigma_S:\Ocal_S\to\Lcal_S$ is the pullback of the global section $\sigma$ of $\Lcal$ defining $V$.
    
    Therefore, by definition of cartesian product of stacks, the objects of the groupoid $[Y^*/\Gm](S)$ are formed by the data of an object $(S\to X,\Mcal,\phi:\Lcal\simeq\Mcal^{\otimes r})$ of $[L^*/\Gm^{\otimes r}]$, an object $(S\to X,\Ncal,\tau)$ of $[\bA^1_X/\Gm]$ and isomorphisms
    \[ \Mcal\simeq\Ncal,\quad \phi\circ\sigma_S = \tau^{\otimes r} \]
    
    This groupoid is clearly equivalent to the one whose objects are maps $S\to X$ together with an invertible sheaf $\Mcal$ over $S$, a global section $\tau$ of $\Mcal$ and an isomorphism $\psi:\Lcal_S\simeq \Mcal^{\otimes r}$ satisfying $\psi(\sigma_S)=\tau^{\otimes r}$. In other terms, we have showed that there is an equivalence of groupoids $[Y^*/\Gm](S)\simeq X_{(V,r)}(S)$. It is easy to check that this equivalence is functorially well behaved, from which the Proposition follows.
\end{proof}

We can use Proposition \ref{prop:root as quotient} to investigate the cohomological invariants of $X_{(V,r)}$, for $X$ smooth. We start with the case of a root stack along an invertible sheaf.
\begin{prop}\label{prop:inv root simple}
    Let $X$ be a smooth scheme over $k$. Consider an invertible sheaf $\Lcal$ over $X$, and let $X_{(\Lcal,r)}$ be the root stack of order $r$ of $X$ along $\Lcal$. Then there is an exact sequence of $\M^{\bullet}(k)$-modules:
    \[ 0 \to \Inv(X,\M) \to \Inv(X_{(\Lcal,r)},\M) \to \ker(c_1(\Lcal))_r \to 0 \]
    where the term on the right stands for the subgroup of $r$-torsion elements in the kernel of $c_1(\Lcal):A^0(X,\M)\to A^1(X,\M)$.
\end{prop}
\begin{proof}
    The open embedding of stacks $[L^*/\Gm^{\otimes r}]\hookrightarrow [L/\Gm^{\otimes r}]$ induces the following long exact sequence of equivariant Chow groups with coefficients:
    \[ 0\to A^0_{\Gm^{\otimes r}}(L,\M) \to A^0_{\Gm^{\otimes r}}(L^*,\M)\to A^0_{\Gm^{\otimes r}}(X,\M) \xrightarrow{\sigma_{0*}} A^1_{\Gm^{\otimes r}}(L,\M) \]
    We have an isomorphism 
    \[A^*_{\Gm^{\otimes r}}(L,\M)\simeq A^*_{\Gm^{\otimes r}}(X,\M)\simeq A^*(X,\M)[\xi]
    \]
    where $\xi$ is the generator of $A^*_{\Gm}(\Spec(k))$ as $\Het(k)$-module. It has codimension $1$ and cohomological degree $1$.
    In particular,
    \[A^1_{\Gm^{\otimes r}}(L,\M)\simeq A^1(X,\M)\oplus A^0(X,\M)\cdot\xi. \] 
    With this identification, the morphism $\sigma_{0*}$ is equal to 
    \[ \alpha \mapsto c_1^{\Gm^{\otimes r}}(\Lcal)(\alpha)=c_1(\Lcal)(\alpha)+r\alpha\cdot\xi .\]
    If $\alpha$ is an element in the kernel of $\sigma_{0*}$, then it must necessarily satisfy
    \[ c_1(L)\cap \alpha=0,\quad r\alpha=0 .\]
    We deduce that $\ker(\sigma_{0*})$ is equal to subgroup of $r$-torsion elements of the kernel of $c_1(\Lcal)$.
    
    Combining what we have just found with the long exact sequence above, the Proposition easily follows.
\end{proof}
Next we study the case of a root stack along a divisor.
\begin{thm}\label{thm:root stack}
    Let $X$ be a smooth scheme over $k$. Consider a Cartier divisor $V\subset X$, and let $X_{(V,r)}$ be the root stack of order $r$ of $X$ along $V$. Then there is an exact sequence of $\M^{\bullet}(k)$-modules:
    \[ 0 \to \Inv(X,\M) \to \Inv(X_{(V,r)},\M) \to \ker(i_*)_r \to 0 \]
    where the term on the right stands for the subgroup of $r$-torsion elements in the kernel of $i_*:A^0(V,\M)\to A^1(X,\M)$.
\end{thm}
\begin{proof}
    Recall from Proposition \ref{prop:root as quotient} that $X_{(V,r)}\simeq [Y^*/\Gm]$: here $Y^*$ stands for the closed subscheme in $L^*\times\bA^1$ of equation $\sigma(x,t)=u^r$, where $\sigma:L\to \bA^1$ is determined by the section $\sigma$ of $\Lcal$ defining $V$ and $u$ is a coordinate for $\bA^1$. The group $\Gm$ acts in the standard way on $\bA^1$ and as $\Gm^{\otimes r}$ on $L$. 
    
    Let $Y$ be the closure of $Y^*$ in $L\times\bA^1$. Observe that we have $Y\smallsetminus Y^*\simeq V$, as it is equal to the intersection of the $\sigma^{-1}(0)$ with the image of the zero section $t_0:X\to L$ (the coordinate $u$ must necessarily be equal to zero). Moreover, the scheme $Y$ is $\Gm$-invariant.
    
    
    
    
    In the following, we denote the root stack $V_{(\Ocal_V(V),r))}$ by $\widetilde{V}$.
    Observe that there is a closed immersion of stacks $\widetilde{V}\hookrightarrow X_{(V,r)}$: the complement of $\widetilde{V}$ in $X_{(V,r)}$ is isomorphic to $X\smallsetminus V$. Therefore we have an exact sequence of Chow groups with coefficients as follows:
    \begin{equation}\label{eq:ex seq root}
        0\to A^0(X_{(V,r)},\M) \to A^0(X_{(V,r)}\smallsetminus \widetilde{V},\M) \simeq A^0(X\smallsetminus V,\M) \to A^0(\widetilde{V},\M)
    \end{equation}
    We need to understand the kernel of the last map in the sequence above.
    
    We claim that there is a factorization 
    \begin{equation}\label{eq:factorization}
    A^0(X_{(V,r)} \smallsetminus \widetilde{V},\M) \simeq A^0(X\smallsetminus V,\M)  \xrightarrow{r\partial_V} A^0(V,\M) \hookrightarrow A^0(\widetilde{V},\M)
    \end{equation}
    where the first map is $r$ times the boundary map, and the second map is the pullback along $\widetilde{V}\to V$, whose injectivity follows from Proposition \ref{prop:inv root simple}.
    
    Let $Z^*$ be the closed subscheme of $Y^*$ where $u=0$, which is isomorphic to $L|_V^*$, and set $U^*:=Y^*\smallsetminus Z^*$. Observe that we have $[Z^*/\Gm^{\otimes r}]\simeq \widetilde{V}$ and $[U^*/\Gm]\simeq X\smallsetminus V$.
    
    
    Observe that $f:Y^*\to L^*$ is a cyclic cover of degree $r$ branched along $Z^*$ and ramified over $L^*|_V$. Moreover, $A^0_{\Gm}(U^*,\M)\simeq A^0(X\smallsetminus V,\M)$. Let $g:L^*|_{X\smallsetminus V}\to X\smallsetminus V$ be the projection morphism. Then it follows from we have just observed that every element of $A^0_{\Gm}(U^*,\M)$ is of the form $f^*g^*\alpha$. By \cite{Rost}*{Def. 1.1.(R3a)} we have:
    \[ \partial_{Z^*}(f^*g^*\alpha) =  r \cdot f_V^* (\partial_{L^*|_V} (g^*\alpha)) \]
    where 
    \[ \partial_{L^*|_V}:A^0_{\Gm}(L^*|_{X\smallsetminus V},\M) \longrightarrow A^0_{\Gm}(L^*|_V,\M) \]
    is the boundary map induced by the closed embedding $L^*|_V\hookrightarrow L^*$, and $f_V:Z^*\to L^*|_V$ is the restriction of $f$ to $Z^*$.
    Consider the cartesian diagram
    \[
    \xymatrix{
    L^*|_V \ar[r] \ar[d]^{g_V} & L^* \ar[d] \\
    V \ar[r] & X
    }
    \]
    Then the compatibility formula implies
    \[ r \cdot \partial_{L^*|_V} (g^*\alpha) = r \cdot g_V^* (\partial_V \alpha ) \]
    where $\partial_V:A^0(X\smallsetminus V,\M)\to A^0(V,\M)$ is the boundary map induced by the closed embedding $V\hookrightarrow X$. Putting all together, we have shown that (\ref{eq:factorization}) holds true.
    
    We deduce from (\ref{eq:ex seq root}) that
    \[ \Inv(X_{(V,r)},\M) = \ker \left(r\cdot\partial_V:A^0(X\smallsetminus V,\M) \longrightarrow A^0(V,\M) \right) \]
    It is easy to see that $A^0(X,\M)$ injects into $\ker(r\partial_V)$. The quotient is then a subgroup of $\ker(A^0(V,\M)\to A^1(X,\M))$, and it is precisely the subgroup of $r$-torsion elements. This concludes the proof.
\end{proof}


\begin{thm}\label{thm:root Nis}
Let $\Xcal$ be a smooth algebraic stack over $k$, and let $r\geq2$ be an integer that is not divisible by ${\rm char}(k)$. Let $V \subset \Xcal$ be a divisor. 

Given any smooth-Nisnevich cover $X\xrightarrow{\pi} \Xcal$ of $\Xcal$ by a scheme, let $V' \subset X$ be the inverse image of $V$. Then $\Inv(\Xcal_{(V,r)},\M)$ is equal to the subset of $\Inv(\Xcal\setminus V,\M)$ of invariants $\alpha$ such that $\partial_{V'}(\pi^* \alpha)$ is of $r$-torsion.
\end{thm}
\begin{proof}
Let $\alpha$ be an invariant of $\Xcal_{(V,r)}$, and let $X$ be any smooth-Nisnevich cover of $\Xcal$. The fiber product $X \times_{\Xcal} \Xcal_{(V,r)}$ is equal to $X_{(V',r)}$, so that $X_{(V',r)}$ is a smooth-Nisnevich cover of $\Xcal_{(V,r)}$. Then by Theorem \ref{thm:root stack} the ramification $\partial_{V'}(\alpha)$ must be of $r$-torsion.

On the other hand, let $\alpha$ be an invariant of $\Xcal \setminus V$ such that $r\partial_{V'}(\alpha)=0$. We claim that $\alpha$ extends to $\Xcal_{(V',r)}$. Again by Theorem \ref{thm:root stack} $\alpha$ defines a cohomological invariant of $X_{(V',r)}$, so we only need to show that the gluing conditions are satisfied. This is trivially true as a cohomological invariant is zero if and only if it is zero on a dense open subset and we know that the invariant glues on the open subset $X \setminus V'$.
\end{proof}

\begin{ex}
Consider the scheme $\mathbb{P}^1$ with two marked points $0, \infty$. It's well known that $\mathbb{P}^1$ has no nontrivial cohomological invariants. If we take the root stack $\mathbb{P}^1_{(0,r)}$, Theorem \ref{thm:root stack} shows us that the cohomological invariants are still trivial, because the complement $\mathbb{P}^1\setminus 0$ is isomorphic to the affine line $\mathbb{A}^1$, which has trivial invariants. 

On the other hand, if we take the divisor $V=0 + \infty$, we have $\mathbb{P}^1 \setminus V \simeq \Gm$ and for every $r$ 
\[
\Inv(\mathbb{P}^1_{(V,r)},\M)=\M^{\bullet}(k) \oplus \alpha \cdot \M^{\bullet}(k)_{r}.
\]
Another way to see this is to note that $\mathbb{P}^1_{(V,r)} \simeq \left[\mathbb{P}^1/\mu_r\right]$, where $\mu_r$ acts by multiplication on either homogeneous coordinate.
\end{ex}

\section{The stack of stable hyperelliptic curves}\label{sec:stable}

A stable hyperelliptic curve $(C,\iota)$ over $\Spec(k)$ is a stable curve $C$ endowed with an involution $\iota:C\to C$ such that the quotient $f:C\to C/\langle \iota \rangle$ is a curve of genus $0$. 

Let $\overline{\Hcal}_g$ be the stack of stable hyperelliptic curves, i.e. the closure of $\Hcal_g$ inside the stack $\overline{\Mcal}_g$ of stable curves. The irreducible components of the boundary divisor $\partial\overline{\Hcal}_g:=\overline{\Hcal}_g\smallsetminus\Hcal_g$ can be described as follows (see \cite{Cor}*{}):
\begin{itemize}
    \item The irreducible divisor $\Delta_0$ is the closure in $\overline{\Mcal}_g$ of the stack $\Delta_0^o$ of irreducible, singular and stable curves $C\to S$ possessing an involution $\iota$ such that the quotient scheme $C/\iota$ is a curve of genus zero.\\
    \item The irreducible divisor $\Delta_i$, for $i=1,\ldots,\lfloor \frac{g}{2} \rfloor$, is the closure in $\overline{\Mcal}_g$ of the stack $\Delta_i^o$ of singular stable curves $C\to S$ whose geometric fibres have two irreducible components of genus $i$ and $g-i$ meeting at one point, and possessing an involution $\iota$ such that the fibres of the quotient scheme $C/\iota$ have two irreducible components of genus $0$ glued at one point.\\
    \item The irreducible divisor $\Xi_i$, for $i=1,\ldots,\lfloor \frac{g}{2}\rfloor -1$, is the closure in $\overline{\Mcal}_g$ of the stack $\Xi_i^o$ of singular stable curves $C\to S$ whose geometric fibres have two irreducible components of genus $i$ and $g-i$ meeting at two points, and possessing an involution $\iota$ such that the fibres of the quotient scheme $C/\iota$ have two irreducible components of genus $0$ glued at one point.
\end{itemize}


Let $\overline{\Hcal}^{o}_g$ be the open substack of $\overline{\Hcal}_g$ such that the restriction of the divisors $\Delta_i$ and $\Xi_i$ in $\overline{\Hcal}_g^o$ coincide with the stacks $\Delta_i^o$ and $\Xi_i^o$ introduced above. 
Observe that the complement of $\overline{\Hcal}_g^o$ in $\overline{\Hcal}_g$ has codimension $2$. 

We give now an alternative description of the stack $\overline{\Hcal}_g^o$ that resembles the one in \cite{ArsVis}*{Cor. 4.7}. Let $\Fcal_g^o$ be the fibred category over the site of $k$-schemes
\[ \Fcal_g^o(S) := \left\{ (P\to S, \Lcal, s) \right\}, \]
where:
\begin{itemize}
    \item $P\to S$ is a family of twisted conics of rank $\geq 2$, i.e. each geometric fibre is a stacky curve whose coarse moduli space is a conic of rank $\geq 2$, and the only stacky structure allowed is at most a $B\mu_2$ at the node (see \cite{ACV}*{Def. 2.1.4} for the notion of twisted curve).
    \item $\Lcal$ is a line bundle on $P$ of degree $-g-1$ whose degree on each irreducible component is $>1$. Moreover, if the fibre possesses a twisted node, the degree of $\Lcal$ restricted to an irreducible component must be an half-integer (more on this a few lines below).
    \item $s$ is a global section of $\Lcal^{\otimes (-2)}$. For every geometric point $z$ of $S$, if the conic $P_{z}$ is smooth than $s$ can have at most one double root, otherwise no double roots are allowed.
\end{itemize}
More explicitly, if $P\to\Spec(\overline{k})$ is a rank $2$ twisted conic, the restriction of $\Lcal$ to an irreducible component can have degree $\frac{r}{2}$ with $r=3,4,\ldots, 2\lfloor \frac{g}{2} \rfloor +1$. Indeed, consider an auxiliary rank $2$ conic $P'$, so that $P$ is the rigidification of $[P'/\mu_2]$ away from the node (here $\mu_2$ acts by sending $(x:y)\mapsto (x:-y)$). A line bundle of bidegree $(\frac{r}{2},\frac{l}{2})$ can then be thought of as a line bundle on $P$ whose lifting to $P'$ has bidegree $(r,l)$.

The category $\Fcal_g^o$ is actually a Deligne-Mumford stack. Observe that $\Fcal_g^o$ contains an open substack isomorphic to $\Hcal_g$, namely the substack of objects $(P\to S,L,s)$ where $P\to S$ is smooth and the vanishing locus of $s$ is \'{e}tale over the base.
\begin{prop}\label{prop:extended description}
    There is an isomorphism $\overline{\Hcal}_g^o\simeq\Fcal_g^o$ that extends the one of \cite{ArsVis}*{Cor. 4.7}.
\end{prop}
\begin{proof}
Let $(C\to S,\iota)$ be an object of $\overline{\Hcal}_g^o$. We can form the stacky quotient $[C/\iota]$, which is by construction a stacky conic of rank $\geq 2$. As the stacky structure is supported on a divisor, we can apply the rigidification process of \cite{ACV}*{5.1} to rigidify $[C/\iota]$ away from the singular locus: the resulting stack $P\to S$ is a family of twisted conics of rank $\geq 2$.

Observe that the morphisms $C\to [C/\iota]$ and $[C/\iota]\to P$ are both flat, hence so must be their composition $f:C\to P$. In particular, this easily implies that $f_*\Ocal_C$ is a locally free sheaf of rank $2$. Moreover, the inclusion $\Ocal_P \hookrightarrow f_*\Ocal_{C}$ of $\iota$-invariant elements has a splitting given by the Reynolds morphism, hence $f_*\Ocal_C\simeq \Ocal_P\oplus \Lcal$ for some line bundle $L$ on $P$.

The structure of $f_*\Ocal_C$ as $\ZZ / 2\ZZ$-graded algebra determines a morphism $\Lcal^{\otimes 2}\to \Ocal_P$ or, equivalently, a global section $s:\Ocal_P\to \Lcal^{\otimes (-2)}$. The vanishing locus of $s$ coincides with the branch locus of $C\to P$, which has degree $2g+2$ by the Hurwitz formula: this implies that $\Lcal$ has total degree $-g-1$. As $C$ is stable, the degree of $\Lcal$ restricted to every irreducible component must be $\geq 2$. Eventually, stability of $C$ also implies that the vanishing locus of $s$ satisfies the required assumptions.

In this way we have constructed a morphism of stacks $\varphi:\overline{\Hcal}_g^o\to \Fcal_g^o$. To construct its quasi-inverse, we can do the following: given an object $(P\to S,\Lcal,s)$ of $\Fcal_g^o$, define $C:=\USpec_{\Ocal_P}(\Ocal_P\oplus\Lcal)$, where $\Ocal_P\oplus\Lcal$ has the $\Ocal_P$-algebra structure induced by $s^{\vee}:\Lcal^{\otimes 2}\to\Ocal_P$. Moreover, there is a well defined involution on $\Ocal_P\oplus\Lcal$ which acts as the identity on $\Ocal_p$ and as $-\id$ on $\Lcal$. This determines an involution $\iota:C\to C$. It is easy to verify that the family of curves $C\to S$ together with $\iota$ is actually an object of $\overline{\Hcal}_g^o$, and that the induced morphism of stacks $\Fcal_g^o\to\overline{\Hcal}_g^o$ gives a quasi-inverse to $\varphi$.
\end{proof}
\begin{rmk}
The necessity of dealing with twisted conics comes from the fact that the morphism $C\to P$, whenever it ramifies over the node,  cannot be flat unless the node is twisted.

On the other hand, the line bundle $\Lcal^{\otimes (-2)}$ comes from the coarse moduli space of $P$, so the global section $s$ does as well. This shows that we still have a morphism from $\overline{\Hcal}_g^o$ to the stack of conics of rank $\geq 2$ with a line bundle $\Mcal$ of degree $2g+2$ and a global section $s$ of $\Mcal$, but this is not an isomorphism.
\end{rmk}
The divisor $\Delta_0^o$ is sent by the isomorphism of Proposition \ref{prop:extended description} to the closed substack in $\Fcal_g^o$ whose objects are triples $(P\to S,\Lcal,s)$ with $P\to S$ smooth. We denote this divisor by $\Dcal_0^o$.

The divisors $\Delta_i^o$ for $i>0$ are isomorphic to the divisors in $\Fcal_g^o$ whose objects  are $(P\to S,\Lcal, s)$ are such that $P$ is a family of conics of rank $2$ and $\Lcal$ has bidegree $(-\frac{2i+1}{2},-\frac{2g-2i+1}{2})$ on each geometric fibre. In particular every fibre of $P\to S$ must have a twist in the node. We denote these divisors by $\Dcal_i^o$.

On the other hand, the divisors $\Xi_i^o$ are isomorphic to the stacks whose objects are triples $(P\to S,\Lcal,s)$ such that $P\to S$ is a family of rank $2$ conics  and $\Lcal$ has bidegree $(-i,-g-i-1)$ on each fibre. These divisors will be denoted $\Xcal_i^o$.

\section{The stack of admissible double coverings}\label{sec:admissible}
In the following, a \emph{nodal curve} is a connected, reduced and proper scheme $C$ over $\Spec(k)$ of dimension $1$ whose singularities are all nodal, i.e. the formal completion of the local ring of $C$ at a singular point $p$ is isomorphic to $k[[x,y]]/(xy)$, for every singular point of $C$.

A \emph{family of nodal curves} is a proper and flat morphism of schemes $C\to S$ such that every geometric fibre is a nodal curve.
\begin{df}[\cite{HM}*{Sec. 4}]\label{def:admissible covering}
An \emph{admissible double covering} of genus $g$ is a finite morphism $C\to P$ between a genus $g$ nodal curve $C$ and a genus $0$ nodal curve $P$ such that:
\begin{itemize}
    \item on the complement of the singular locus of $P$, the morphism $C\to P$ has degree $2$.
    \item The image of every node $q$ of $C$ is a node $p$ of $P$, and we have that
    \[ \widehat{\Ocal_{V,p}}\simeq k[[x,y]]/(xy) \longrightarrow \widehat{\Ocal_{C,q}}\simeq k[[u,v]]/(u,v) \]
    sends $x\mapsto u^e$ and $y\mapsto v^e$ with $1\leq e \leq 2$.
    \item Only a finite number of automorphisms of $P$ fixes both the ramification divisor of $C\to P$ and the singular locus.
\end{itemize}
Given a scheme $S$, a \emph{family of admissible double coverings} of genus $g$ over $S$ is an $S$-morphism $C\to P$ between families of nodal curves over $S$ such that over every geometric point $\overline{s}$ of $S$ the morphism $C_{\overline{s}}\to P_{\overline{s}}$ is an admissible covering of degree $2$ and genus $g$.
\end{df}
\begin{rmk}
The definition above differs from the one of \cite{HM} because we do not require to have $2g+2$ sections $\sigma_i:\Spec(k)\to P$ such that the ramification divisor, away from the nodes, can be written as $\sum \sigma_i$.
\end{rmk}

We can form the fibred category $\widetilde{\Hcal}_g$ of admissible double coverings of genus $g$, that is
\[ \widetilde{\Hcal}_g (S) := \left\{ (C \to P \to S) \right\},  \]
where a morphism $(C\to P\to S)\to (C'\to P'\to S')$ consists of a morphism $S\to S'$ and two isomorphisms $C\simeq C'_S$ and $P\simeq P'_S$ which commute with the covering maps.

The stack $\widetilde{\Hcal}_g$ is a smooth Deligne-Mumford stack. One way to see this is to consider the stack $\Bcal_{0,2g+2}(S_{2})$ of twisted admissible coverings (see \cite{ACV}*{2.2}), which is isomorphic to the stack of admissible coverings of degree $2$ of Harris-Mumford by \cite{ACV}*{Prop. 4.2.2}. The stack $\Bcal_{0,2g+2}(S_2)$ is a smooth Deligne-Mumford stack, and the automorphism group of every twisted $S_2$-admissibe covering contains the symmetric group $S_{2g+2}$, which acts by permutation on the markings. By \cite{ACV}*{5.1} we can take the rigidification $[\Bcal_{0,2g+2}(S_2)/\Bcal S_{2g+2}]$, and it is easy to see that this last stack is isomorphic to $\widetilde{\Hcal}_g$.

We can recast Definition \ref{def:admissible covering} using the notion of admissible involution.
\begin{df}[\cite{Bea}*{Sec. 3}]\label{df:admissible inv}
    An \emph{admissible involution} of a nodal curve $C$ over $\Spec(k)$ is an involution $\iota:C\to C$ such that:
    \begin{itemize}
        \item no irreducible component is fixed by $\iota$.
        \item The quotient $C/\langle \iota\rangle$ is a nodal curve of genus $0$.
        \item If a node is fixed by $\iota$, then $\iota$ acts on the local ring $\Ocal_{C,p}$, hence on its formal completion $\widehat{\Ocal_{C,p}}$, and it does not switch the two branches of the node.
    \end{itemize}
\end{df}
The Definition above extends in an obvious way to families of nodal curves.

\begin{prop}\cite{Sca}*{Prop. 3.3.7}\label{prop:cov=inv}
    Let $C\to P\to S$ be a family of admissible double coverings over $S$. Then there exists a unique admissible involution $\iota:C\to C$ over $S$ such that $P=C/\langle \iota \rangle$.
    
    Viceversa, given an admissible involution $\iota:C\to C$ of a family of nodal curves, the morphism $C\to C/\langle \iota \rangle$ is an admissible double covering.
\end{prop}
The following results give a somewhat more concrete description of how an admissible double covering looks like.
\begin{prop}\cite{Sca}*{Lemma 4.1.2, Prop. 4.1.3}
    Let $C\to P$ be an admissible double covering, and let $\iota:C\to C$ be the admissible involution of Proposition \ref{prop:cov=inv}. Then:
    \begin{enumerate}
        \item the nodal curve $C\to S$ is semistable. 
        \item Every irreducible component of $C$ is smooth. Moreover, if $\iota(C_i)=C_i$, then $C_i\to C_i/\langle\iota\rangle$ is a hyperelliptic curve, otherwise $C_i$ and $\iota(C_i)$ are two disconnected projective lines.
        \item Every non-stable component is a projective line intersecting the rest of the curve in two points, which get switched by $\iota$.
    \end{enumerate}
\end{prop}
An immediate consequence of the Proposition above is that the admissible involution $\iota$ of an admissible double covering $C\to P$ descends to a hyperelliptic involution $\overline{\iota}$ of the stable model $\overline{C}$ of $C$: in other terms, the stable curve $\overline{C}$ is a stable hyperelliptic curve. This argument works for families too.
\begin{cor}
    Let $C\to P\to S$ be a family of admissible double covering, and let $\overline{C}\to S$ be its stable model. Then the admissible involution $\iota:C\to C$ descends to a hyperelliptic involution $\overline{\iota}:\overline{C}\to\overline{C}$, and the pair $(\overline{C}\to S,\overline{\iota})$ is a family of stable hyperelliptic curves.
\end{cor}
The stabilization morphism commutes with base change, hence the Corollary above determines a morphism of stacks
\[ \varphi: \widetilde{\Hcal}_g \longrightarrow \overline{\Hcal}_g,\quad (C\to P\to S)\longmapsto (\overline{C}\to S,\overline{\iota}) \]
called \emph{the collapsing morphism} \cite{Sca}*{Def. 4.1.13}. 

Observe that the collapsing morphism induces a bijection at the level of geometric points, and it is obviously an isomorphism over the open substack of smooth curves. This in turn implies that the irreducible components of the boundary $\widetilde{\Hcal}_g\smallsetminus \overline{\Hcal}_g$ are $\widetilde{\Delta}_i:=\varphi^{-1}(\Delta_i)$ for $i=0,1,\ldots,\lfloor\frac{g}{2}\rfloor$ and $\widetilde{\Xi}_j:=\varphi^{-1}(\Xi_j)$ for $j=1,2,\ldots,\lfloor \frac{g}{2}\rfloor$.

Recall that in Section \ref{sec:stable} we defined the open substack  $\overline{\Hcal}_g^o$ of $\overline{\Hcal}_g$. Set then $\widetilde{\Hcal}_g^o:=\varphi^{-1}(\overline{\Hcal}_g^o)$. As for the case of stable hyperelliptic curves, the complement of $\widetilde{\Hcal}_g^o$ in $\widetilde{\Hcal}_g$ has codimension $2$, and the irreducible components of $\widetilde{\Hcal}_g^o\smallsetminus \Hcal_g$ are $\widetilde{\Delta}_i^o:=\varphi^{-1}(\Delta_i^o)$ and $\widetilde{\Xi}_j^o:=\varphi^{-1}(\Xi_j^o)$.

Observe that the collapsing map induces an isomorphism of stacks for almost all the irreducible components of the boundary of $\widetilde{\Hcal}_g^o$ and $\overline{\Hcal}_g^o$, with the only exception of $\widetilde{\Delta}_0^o\to \Delta_0^o$: indeed, compared to its stable model, an admissible covering in $\widetilde{\Delta}_0^o$ has an additional automorphism, namely the involution of the projective line that switches the two intersection points.
\begin{prop}\label{prop:is a root}
    The stack $\widetilde{\Hcal}_g^o$ is a root stack of $\overline{\Hcal}_g^o$ of order $2$ along the divisor $\Delta_0^o$.
\end{prop}
\begin{proof}
    It follows directly from \cite{Sca}*{Prop. 4.4.5}
\end{proof}

\section{Computation of cohomological invariants}\label{sec:computations}
We start by recalling some facts on cohomological invariants of the stack of smooth hyperelliptic curves. A more in depth discussion of this subject can be found in the Introduction to \cite{DilPir2}.

Let ${\rm \'et}_n$ be the stack of \'etale algebras of degree $n$. It is isomorphic to the classifying stack ${\rm BS}_n$. Given a family of smooth hyperelliptic curves $C\to S$ with involution $\iota$, the ramification divisor of $C\to C/\langle \iota \rangle$ is the spectrum of an \'{e}tale $\Ocal_S$-algebra of degree $2g+2$. This defines a map $$\Hcal_g \longrightarrow {\rm \'et}_{2g+2} \simeq {\rm BS}_{2g+2}$$ which is smooth-Nisnevich, so that the pullback $\Inv({\rm BS}_{2g+2},\M) \to \Inv(\Hcal_g,\M)$ is injective. The cohomological invariants of ${\rm BS}_{2g+2}$ are isomorphic to
\[
\M^{\bullet}(k) \oplus \alpha_1\cdot \M^{\bullet}(k)_2 \oplus \ldots \oplus \alpha_{g+1}\cdot \M^{\bullet}(k)_2
\]
where the classes $\alpha_i$ are the \emph{arithmetic Stiefel-Whitney classes} coming from the invariants of ${\rm BO}_{2g+2}$, see \cite{GMS}*{23.3} for the case $\M=\H_D$, and \cite{DilPir2}*{4.9} for the general case.

\begin{thm}[\cite{DilPir2}*{Thm. 6.9, Rmk. 6.10}]\label{thm:DilPir2}
Let $\M$ be an $\ell$-torsion cycle module, and let $I_g$ be the submodule of $\Inv({\rm BS}_{2g+2},\M)$ given by 
\[
I_g=\M^{\bullet}(k) \oplus \bigoplus_{i=2}^{g+1} \alpha_i \cdot \M^{\bullet}(k)_2 .
\]
Let $n_g$ be equal to $4g+2$ if $g$ is even and $8g + 4$ if $g$ is odd. 
If $g$ is even, we have
\[
\Inv(\Hcal_g,\M) = \alpha'_1\cdot \M^{\bullet}(k)_{n_g} \oplus I_g \oplus \beta_{2g+2}\cdot \M^{\bullet}(k)_2
\]
where $\alpha'_1$ and $\beta_{2g+2}$ are elements of degree respectively $1$ and $2g+2$.

If $g$ is odd, there is an exact sequence
\[
0 \to {\alpha'_1}\cdot \M^{\bullet}(k)_{n_g} \oplus I_g \oplus w_2\cdot \M^{\bullet}(k)_2 \to \Inv(\Hcal_g,\M) \to \M^{\bullet}(k)_2
\]

where $\alpha'_1$ and $w_2$ are elements of degree respectively $1$ and $2$ and the last map lowers degree by $g+2$. When $k=\overline{k}$, the last map is surjective and the sequence splits.
\end{thm}

As the pullback of an open immersion is injective, the invariants of $\widetilde{\Hcal}_g$ inject into those of $\Hcal_g$. We want to understand which cohomological invariants extend to the stack $\widetilde{\Hcal}_g^o$ introduced in Section \ref{sec:admissible}, which in turn has the same invariants as $\widetilde{\Hcal}_g$ as they differ by a subset of codimension $2$.

\begin{prop}\label{prop:inv Hgtbd}
We have:
\begin{enumerate}
    \item for every $g\geq 2$ and for every $\ell$-torsion cycle module $\M$, the cohomological invariants of $\Hcal_g$ with coefficients in $\M$ that come from ${\rm B}S_{2g+2}$ extend to cohomological invariants of $\overline{\Hcal}^o_g\smallsetminus\Delta_0^o$.
    \item For $g\geq 3$ odd, the invariant $w_2$ pulled back along $\Hcal_g\to {\rm B}\PGL_2$ does not extend to an invariant of $\overline{\Hcal}^o_g\smallsetminus\Delta_0^o$.
\end{enumerate}
\end{prop}
\begin{proof}
To prove (1), we show that the map $\Hcal_g\to {\rm B}S_{2g+2}$ extends to the whole of $\overline{\Hcal}^o_g\smallsetminus\Delta_0^o$.
Because of the isomorphism $\overline{\Hcal_g}^o \simeq \Fcal_g^o$ of Proposition \ref{prop:extended description}, we can equivalently prove that there is a morphism $\Fcal_g^o\smallsetminus\Dcal_0^o\to {\rm B}S_{2g+2}$ extending the morphism on $\Hcal_g$.

This last task is readily accomplished: to every object $(P\to S,L,s)$ of $\Fcal_g^o\smallsetminus\Dcal_0^o$ we associate the vanishing locus of $s$, which by construction is \'{e}tale over $S$ of degree $2g+2$, determining in this way the desired map $\overline{\Hcal}^o_g\smallsetminus\Delta_0^o\to {\rm B}S_{2g+2}$.

Point (2) has already been shown in the proof of \cite{DilPir}*{Thm. A.1}.
\end{proof}

\begin{prop}\label{prop:inv Hgtbd trivial case}
    Let $\M$ be a $p^n$-primary torsion cycle module, with $2, {\rm char}(k)\nmid p$. Then
    \[\Inv(\overline{\Hcal}_g^o\smallsetminus\Delta_0^o,\M)\simeq\M^{\bullet}(k). \]
\end{prop}
\begin{proof}
    We know from \cite{DilPir2}*{Prop. 7.5} that
    \begin{equation}\label{eq:ptorsion}
    \Inv(\Hcal_g,\M)\simeq \M^{\bullet}(k)\oplus\M^{\bullet}(k)_{p^m}[1] 
    \end{equation}
    as $\M^{\bullet}(k)$-modules, where $p^m={\rm gcd}(p^n,2g+1)$. 
    
    Let $[\Lcal]$ be the generator of the $p^n$-torsion subgroup of \[\Pic(\Hcal_g)\simeq \ZZ/2^{1+r_g}(2g+1)\ZZ,\] 
    where $r_g\in\{0,1\}$ is such that $r_g\equiv g$ mod $2$. Recall \cite{DilPir2}*{Lm. 1.13} that ${\rm Inv}^1(\Hcal_g,\K_\ell) = {\rm H}^1_{\rm ét}(\Hcal_g, \mu_\ell)$. Then in (\ref{eq:ptorsion}) above we have
    \[ \M^{\bullet}(k)_{p^m}[1]\simeq \M^{\bullet}(k)_{p^{m}}\cdot [\Lcal], \]
    with the product induced by the $\K_{\ell}$-module structure of $\M$.
    
    First suppose that $\M=\K_{p^n}$. In order to prove that $\overline{\Hcal}_g^o\smallsetminus\Delta_0^o$ has only trivial invariants, it is enough to check that its Picard group has trivial $p^n$-torsion: this easily follows from the computation of $\Pic(\overline{\Hcal}_g)$ (see \cite{Cor}*{Thm. 2} in characteristic $0$ and \cite{DilPir}*{Thm. B.1} in positive characteristic). Moreover, this also shows that the ramification of $[\Lcal]$ along the boundary divisor $\cup \Delta_i^o \cup \Xi_j^o$, regarded as an element of $\ZZ/p^m\ZZ$, is invertible.
    
    For a general $p^n$-primary torsion cycle module $\M$ and $\gamma\in \M^{\bullet}_{p^m}(k)$, let $[\Lcal]\cdot \gamma$ be an invariant of $\Hcal_g$ that extends to an invariant of $\overline{\Hcal}_g^o\smallsetminus\Delta_0^o$: its ramification is equal to $\gamma$ times the ramification of $[\Lcal]$, which we know to be invertible as an element of $\ZZ/p^m\ZZ$. This implies $\gamma=0$, and we are done.
\end{proof}
\begin{cor}
    Let $\M$ be a $p^n$-primary torsion Galois module, with ${\rm char}(k), 2\nmid p$. Then
    \[ \Inv(\widetilde{\Hcal}_g,\M)\simeq\M^{\bullet}(k) \]
\end{cor}
\begin{proof}
    The stack $\overline{\Hcal}_g^o\smallsetminus\Delta_0^o$ is an open substack of $\widetilde{\Hcal}_g$, hence the invariants of the latter inject into the the ring of invariants of the former, which is trivial by Proposition \ref{prop:inv Hgtbd trivial case}.
\end{proof}

\subsection{The $g$ even case}

\begin{prop}\label{prop:exceptional}
    Let $g\geq 2$ be even, and let $\M$ be a $2^n$-primary torsion Galois module. Then the exceptional invariant $\beta_{g+2}$ of $\Hcal_g$ extends to $\overline{\Hcal}^o_g\smallsetminus\Delta_0^o$.
\end{prop}
\begin{proof}
We give the proof for $\M=\K_2$, the general case follows easily from this one.

Let us recall how $\beta_{g+2}$ is constructed. Let $\Ccal\to\Hcal_g$ be the universal smooth conic, and let $\Lcal$ and $\sbold$ be respectively the universal line bundle on $\Ccal$ and the universal global section $\sbold$ of $\Lcal^{\otimes (-2)}$. The existence of such objects over $\Hcal_g$ is a consequence of \cite{ArsVis}*{Prop. 2.2}.

Let $\Ucal$ be the complement in $\Ccal$ of the vanishing locus $\Vcal$ of $\sbold$: in particular, we have that $\Lcal^{\otimes 2}$ is trivial on $\Ucal$ or equivalently $\Lcal_{\Ucal}$ belongs to ${\rm H}^1(\Ucal,\mu_2)$.

Then we can define a degree $1$ cohomological invariant $t$ of $\Ucal$ as follows: given a morphism $p:\Spec(K)\to\Ucal$ we set $t(p)=[\Lcal_K]$, which belongs to ${\rm H}^1(\Spec(K),\mu_2)$.

Another way to see the invariant $t$ is the following. Given a morphism $p:\Spec(K)\to\Ucal$, consider any lifting $p':\Spec(K) \to \Lcal^{\otimes (-2)}$. The element $p'^*(\sbold) \in K^*$ is well defined up to squares, and it's easy to see that this glues to the cohomological invariant $t$. In particular, this shows that the boundary of $t$ in the complement of $\Ucal$ is equal to $1$.

This invariant does not extend to a global invariant of $\Ccal$, but $\alpha_{g+1}\cdot t$ does, where $\alpha_{g+1}$ is the degree $g+1$ generator of $\Inv({\rm B}S_{2g+2},\M)$. The reason is that 
\[ \partial_{\mathcal{V}}(t \cdot \alpha_{g+1})=\alpha_{g+1} \]
and the restriction of $\alpha_g$ to $\mathcal{V}$ is zero. This is explained in detail in \cite{DilPir}*{Sec. 2}.

Using the well known property that $\Ccal$ is actually the projectivization of a rank $2$ vector bundle, we deduce that $\Inv(\Ccal)\simeq\Inv(\Hcal_g)$, hence the construction above gives the exceptional invariant $\beta_{g+2}$ of $\Hcal_g$.

The construction of the exceptional invariant for $\overline{\Hcal}_g^o\smallsetminus\Delta_0^o$ proceeds along the same lines.
Let $\Pcal_{sm}\to \overline{\Hcal}_g^o\smallsetminus\Delta_0^o$ be the smooth locus of the universal conic, and let $\Lcal$ and $\sbold$ be respectively the universal line bundle restricted to $\Pcal_{sm}$ and the universal global section of $\Lcal^{\otimes (-2)}$. Again, the existence of such objects is assured by the isomorphism $\overline{\Hcal}_g^o\smallsetminus\Delta_0^o\simeq\Fcal_g^o$ of Proposition \ref{prop:extended description}.

Observe that $\Pcal_{sm}\to \overline{\Hcal}_g^o\smallsetminus\Delta_0^o$ is a smooth-Nisnevich covering of $\overline{\Hcal}_g^o\smallsetminus\Delta_0^o$, hence the cohomological invariants of the latter inject into $\Inv(\Pcal^{sm}, \M)$ via pullback.

Moreover, if an invariant of $\Ccal$, the universal smooth conic, can be extended to $\Pcal_{sm}$, then it necessarily comes from $\overline{\Hcal}_g^o\smallsetminus\Delta_0^o$. Indeed, let $\xi$ be such an invariant: due to  $\Pcal_{sm}\to\overline{\Hcal}_g^o\smallsetminus\Delta_0^o$ being a smooth-Nisnevich cover, we have that $\xi$ comes from $\overline{\Hcal}_g^o\smallsetminus\Delta_0^o$ if and only if it glues, i.e. $\pr_1^*\xi-\pr_2^*\xi=0$ in $\Inv(\Pcal_{sm}\times_{\overline{\Hcal}_g^o\smallsetminus\Delta_0^o}\Pcal_{sm}, \M)$.

On the other hand the restriction of $\pr_1^*\xi-\pr_2^*\xi$ to $\Ccal\times_{\Hcal_g}\Ccal$ is zero, as we already know that the invariant glues over $\Hcal_g$. As $\Ccal\times_{\Hcal_g}\Ccal$ is an open substack of $\Pcal_{sm}\times_{\overline{\Hcal}_g^o\smallsetminus\Delta_0^o}\Pcal_{sm}$, the invariants of the latter injects into the invariants of the former. Therefore $\pr_1^*\xi-\pr_2^*\xi$ is zero as well, hence $\xi$ glues to an invariant of $\overline{\Hcal}_g^o\smallsetminus\Delta_0^o$, as claimed.

The same argument used to produce the invariant $\beta_{g+2}$ of $\Ccal$ applies in this situation, giving us an invariant on $\Pcal_{sm}$ which obviously extends $\beta_{g+2}$. More precisely, one considers the complement of the vanishing locus of $\sbold$ in $\Pcal_{sm}$: here we have the invariant $t$ of cohomological degree $1$ given by
\[ t= [\Lcal]\in {\rm H}^1(\Pcal_{sm}\smallsetminus\{\sbold=0\},\mu_2). \]
The product $\alpha_{g+1}\cdot t$, due to the very same reason for which $\alpha_{g+1}\cdot t$ extended in the smooth case, extends to a global invariant of $\Pcal_{sm}$. With a little abuse of notation, this invariant is also denoted $\beta_{g+2}$. Then, for what we observed a few lines above, the invariant $\beta_{g+2}$ must come from $\overline{\Hcal}_g^o\smallsetminus\Delta_0^o$. This concludes the proof.
\end{proof}
Putting together Proposition \ref{prop:inv Hgtbd} and Proposition \ref{prop:exceptional}, we derive the following result.
\begin{cor}\label{cor:inv Hgtbd}
    Let $g\geq 2$ be an even integer and $\M$ a $2$-primary torsion cycle module. 
    Then we have
           \[ \Inv(\Hcal_g,\M)\simeq \Inv(\overline{\Hcal}_g^o\smallsetminus\Delta_0^o,\M) \]
\end{cor}
We can now apply the formula for the invariants of a root stack to compute the invariants of $\widetilde{\Hcal}_g$.
\begin{thm}\label{thm:inv even}
    Let $g\geq 2$ be even and $\M$ a cycle module of $p^n$-primary torsion, with $p\neq c$ prime. Then if $p=2$ we have
    \[ \Inv(\widetilde{\Hcal}_g,\M)\simeq \Inv(\Hcal_g,\M), \]
    otherwise the invariants of $\widetilde{\Hcal}_g$ are trivial.
\end{thm}
\begin{proof}
    The complement in $\widetilde{\Hcal}_g$ of the open substack $\widetilde{\Hcal}_g^o$ has codimension $2$, hence these two stacks have the same cohomological invariants.
    
    We know from Proposition \ref{prop:is a root} that $\widetilde{\Hcal}_g^o$ is a root stack of order $2$ of $\overline{\Hcal}_g^o$ along the divisor $\Delta_0^o$. Therefore, by Theorem \ref{thm:root stack}, the cohomological invariants of $\widetilde{\Hcal}_g^o$ are equal to the cohomological invariants of $\overline{\Hcal}_g^o\smallsetminus\Delta_0^o$ whose ramification along $\Delta_0^o$ is of $2$-torsion.
    
    It follows from Corollary \ref{cor:inv Hgtbd} that the ramification of all the cohomological invariants of $\overline{\Hcal}_g^o\smallsetminus\Delta_0^o$ along $\Delta_0^o$ is of $2$-torsion. This concludes the proof.
\end{proof}
\begin{rmk}
    Theorem \ref{thm:inv even} should be compared with \cite{DilPir}*{Thm. A.1}, where it is shown that the cohomological invariants of $\overline{\Hcal}_g$ are all trivial.
    This does not come as a surprise: the compactification of $\Hcal_g$ by admissible double coverings preserves most of the structure needed to construct the cohomological invariants of smooth hyperelliptic curves. On the other hand, this is not the case for the compactification by stable hyperelliptic curves.
\end{rmk}

\subsection{The $g$ odd case}
When $g$ is odd and $\M$ is a $2^n$-primary torsion cycle module with $n\geq 2$, the stack $\Hcal_g$ has basically two more invariants than in the even genus case (see Theorem \ref{thm:DilPir2}). One of the two is the invariant $w_2$ of cohomological degree $2$ pulled back from ${\rm B}\PGL_2$, and the other comes from the $4$-primary torsion of $\Pic(\Hcal_g)$, which for $g$ odd is equal to $\ZZ/4(2g+1)\ZZ$. This second invariant, denoted $\alpha_1'$, is a square root of $\alpha_1$, the degree one invariant coming from ${\rm B}S_{2g+2}$.

We have already seen in Proposition \ref{prop:inv Hgtbd} that $w_2$ does not extend to an invariant of $\overline{\Hcal}_g^o\smallsetminus\Delta_0^o$, so in order to have a complete picture in degree $\leq g+1$ of $\Inv(\overline{\Hcal}_g^o\smallsetminus\Delta_0^o,\M)$ we only need to understand if $\alpha_1'$ extends or not.
A bit surprisingly, the value of the genus plays a role here.
\begin{prop}
    Let $g\geq 3$ be an odd integer and let $\M$ be a $2$-primary torsion cycle module. There are submodules ${\rm N}^{\bullet}_g$ of $\M^{\bullet}(k)_2$ such that: 
    \begin{enumerate}
        \item If $g\geq 5$, we have an exact sequence
        \[ 0 \rightarrow \Inv({\rm B} S_{2g+2},\M) \rightarrow \Inv(\overline{\Hcal}_g^o\smallsetminus\Delta_0^o,\M) \rightarrow  {\rm N}^{\bullet}_g \rightarrow 0\]
        such that the inverse image of a non-zero element in ${\rm N}^{\bullet}_g$ has degree at least $g+2$.
        \item If $g=3$, define $$I = \alpha_2 \cdot \M^{\bullet}(k)_2 \oplus \alpha_3 \cdot \M^{\bullet}(k)_2 \oplus \alpha_4 \cdot \M^{\bullet}(k)_2 \subset \Inv({\rm BS}_4,\M).$$ 
        Then there is an exact sequence
        $$0 \rightarrow  \M^{\bullet}(k) \oplus \M^{\bullet}(k)_4\!\left[1\right] \oplus I \rightarrow \Inv(\overline{\Hcal}_3^o\smallsetminus\Delta_0^o,\M) \rightarrow {\rm N}^{\bullet}_3 \rightarrow 0$$
        such that the inverse image of a non-zero element in $K_3$ has degree at least $5$. The $\M^{\bullet}(k)_4\!\left[1\right]$ is equal to $\alpha'_1 \cdot \M^{\bullet}(k)_4$, where $\alpha'_1$ is a square root of $\alpha_1$.
    \end{enumerate}
\end{prop}
\begin{proof}
    First suppose $\M=\K_{2^n}$, $n\geq 2$.
    
    When $g$ is odd the Picard group of $\Hcal_g$ is generated by the Hodge line bundle $\lambda$ \cite{GV}. Consequently, $\Pic(\overline{\Hcal}_g)$ is generated by $\lambda$, which extends to the compactification, and the boundary divisor classes $\left[ \Xi_j \right], \left[\Delta_i\right]$.
    
    In \cite{CH}, Cornalba and Harris showed that, over a field of characteristic zero, the group $\Pic(\overline{\Hcal}_g) \otimes \mathbb{Q}$ is defined by the single relation
    
   \[
   (8g+4)\lambda = g \left[ \Delta_0 \right] + 2\sum_{j=1}^{\lfloor \frac{g-1}{2} \rfloor} (j+1)(g-j)[\Xi_j] + 4\sum_{i=1}^{\lfloor \frac{g}{2} \rfloor} i(g-i)[\Delta_i].
   \]

    The same result was proven in positive characteristic by Yamaki \cite{Yam}*{Thm. 1.7}. From \cite{Cor}*{Thm. 2} in characteristic zero and \cite{DilPir}*{Thm. B.1} in positive characteristic $c > 2$, we moreover know that the Picard group of $\overline{\Hcal}_g$ is torsion free, and thus free of rank $g$.
    
    Consider the restriction map $\Pic(\overline{\Hcal}_g^o)_4 \rightarrow \Pic(\Hcal_g)_4=\ZZ/4\ZZ$. Using the equality ${\rm H}^1(\Xcal,\mu_\ell)={\rm Inv}^1(\Xcal,\K_\ell)$ \cite{DilPir2}*{Lm. 1.13} and the Kummer exact sequence we get a commutative diagram
    
    \[
    \xymatrix{
    \mathcal{O}^*(\overline{\Hcal}_g^o\!\smallsetminus \!\Delta_0^o)/\mathcal{O}^*(\overline{\Hcal}_g^o\!\smallsetminus \!\Delta_0^o)^{4} \ar[r] \ar[d] & {\rm Inv}^1(\overline{\Hcal}_g^o\!\smallsetminus \!\Delta_0^o,\K_4) \ar[d] \ar[r] & \Pic(\overline{\Hcal}_g^o\!\smallsetminus \!\Delta_0^o)_4 \ar[r] \ar[d] & 0\\
   \mathcal{O}^*({\Hcal}_g)/\mathcal{O}^*({\Hcal}_g)^{4} \ar[r] & {\rm Inv}^1({\Hcal}_g,\K_4) \ar[r] & \Pic({\Hcal}_g)_4 \ar[r] & 0
    }
    \]
    
    The rows are exact sequences, the first vertical arrow is an isomorphism (both are equal to $k^*/(k^*)^4$), and the second vertical arrow is injective. We conclude that a primitve $4$-torsion element in $\Pic(\overline{\Hcal}_g^o\smallsetminus \Delta_0^o)$, if any, would have to map to a generator of the $4$-torsion of $\Pic(\Hcal_g)$, which is generated by $(2g+1)\lambda$. In particular, if there is a primitive $4$-torsion element in $\Pic(\overline{\Hcal}_g^o\smallsetminus \Delta_0^o)$ we may get one in the form $(2g+1)\lambda + \beta$, where $\beta$ is a linear combination of the boundary divisors.
    
    If $g=3$, it is immediate to verify that $(2g+1)\lambda + 2\left[ \Xi_1 \right] + 2\left[\Delta_1\right]$ is a primitive $4$-torsion element. Then the description of $\Inv(\Hcal_3,D)$ of Theorem \ref{thm:DilPir2}, combined with Proposition \ref{prop:inv Hgtbd} and the observations above, yields the desired result.
    
    If $g > 3$, let $(2g+1)\lambda + \beta$ be a $4$-torsion element. Then 
        \[
    (8g+4)\lambda - (8g+4)\lambda - 4\beta =  2\sum_{j=1}^{\lfloor \frac{g-1}{2} \rfloor} (j+1)(g-j)[\Xi_j] + 4\sum_{i=1}^{\lfloor \frac{g}{2} \rfloor} i(g-i)[\Delta_i].
    \]
    
    Note that the coefficient of $\left[\Xi_2\right]$ on the RHS is equal to $6(g-2)$, which is not divisible by $4$. This gives a non-trivial relation between the $\left[\Xi_i\right]$ and $\left[\Delta_j\right]$ classes, which is impossible as we know that $\Pic(\overline{\Hcal}_g^o)\otimes \mathbb{Q}$ has rank $g-1$.
    
    What we have just shown implies that for $g>3$ the ramification of $\alpha_1'$ along the boundary divisor is of $2$-torsion, because we know that $\alpha_1=2\alpha_1'$ can be extended: this implies that an invariant $\gamma\cdot\alpha_1'$ of $\M^{\bullet}(k)_4\cdot\alpha_1'$ is unramified if and only if $\gamma$ is a multiple of $2$, hence the subgroup of $\M^{\bullet}(k)_4\cdot\alpha_1'$ of unramified elements coincides with $\M^{\bullet}(k)_2\cdot\alpha_1$.
    
    As in the genus three case, combining this with Proposition \ref{prop:inv Hgtbd} and Theorem \ref{thm:DilPir2}, we get the claimed description.
\end{proof}
This allows us to give an almost complete description of the cohomological invariants of $\widetilde{\Hcal}_g$ for $g$ odd.
\begin{thm}\label{thm:inv odd}
Let $g\geq 3$ be an odd integer and let $\M$ be a $p^n$-primary torsion cycle module, with $p\neq c$ prime. If $p=2$ there is a (possibly trivial) submodule ${\rm N}^{\bullet}_g$ of $\M^{\bullet}(k)_2$ and an exact sequence
    \[ 0 \rightarrow \Inv({\rm B} S_{2g+2},\M) \rightarrow \Inv(\widetilde{\Hcal}_g,\M) \rightarrow  {\rm N}^{\bullet}_g \rightarrow 0\]
such that the inverse image of a non-zero element in ${\rm N}^{\bullet}_g$ has degree at least $g+2$. If $p>2$, then the cohomological invariants of $\widetilde{\Hcal}_g$ are trivial.
\end{thm}
\begin{proof}
    For $g\geq 5$, we can apply exactly the same argument that we used to prove Theorem \ref{thm:inv even}. 
    
    In the case $g=3$, the ramification of the invariant $\alpha_1'$ along $\Delta_0$ cannot be of $2$-torsion, otherwise we would be able to extend the invariant $\alpha_1=2\alpha_1'$ to $\overline{\Hcal}_g$. Therefore, this case can be worked out as in the proof of Theorem \ref{thm:inv even} as well.
\end{proof}

\subsection{Brauer groups}

Theorem \ref{thm:inv even} and Theorem \ref{thm:inv odd} give to us all the informations needed to determine the prime-to-the-characteristic part of the Brauer group of $\widetilde{\Hcal}_g$.

\begin{thm}\label{thm:Brauer}
    Let $g\geq 2$ and set $c={\rm char}(k)$. Let $^c{\rm Br}(-)$ denote the prime-to-$c$ part of the Brauer group. Then:
    \[ ^c{\rm Br}(\widetilde{\Hcal}_g) \simeq \prescript{c}{}{\rm Br}(k) \oplus H^1(k,\ZZ/2\ZZ)\oplus\ZZ/2\ZZ. \]
    In particular, when $c=0$ the formula above describes the whole Brauer group of $\widetilde{\Hcal}_g$.
\end{thm}

\begin{proof}
    The proof works exactly as in \cite{DilPir2}*{Thm. 8.1}. The starting point is the isomorphism
    \begin{equation}\label{eq:iso inv br}
         {\rm Inv}^2(\widetilde{\Hcal}_g,\ZZ/\ell\ZZ(-1))\simeq {\rm Br}'(\widetilde{\Hcal}_g)_\ell,
    \end{equation}
    where on the right we have the $\ell$-torsion part of the cohomological Brauer group.
    
    Theorem \ref{thm:inv even} and Theorem \ref{thm:inv odd} can then be used to compute the left hand side of (\ref{eq:iso inv br}): one has only to observe that all the elements in the cohomological Brauer group that do not come from the ground field are of $2$-primary torsion.
    
    To conclude, by results of Edidin-Hasset-Kresch-Vistoli and Kresch-Vistoli \cite{DilPir2}*{Thm.1.1}, we have that the prime-to-${\rm char}(k)$ part of the cohomological Brauer group of a smooth, quasi-projective, tame Deligne Mumford stack belongs to the Brauer group, so in particular $^c{\rm Br}(\widetilde{\Hcal}_g)= {^c}{\rm Br}'(\widetilde{\Hcal}_g)$.
\end{proof}


\begin{bibdiv}
	\begin{biblist}
	    \bib{ACV}{article}{
               author={Abramovich, D.},
               author={Corti, A.},
               author={Vistoli, A.},
               title={Twisted bundles and admissible covers},
               note={Special issue in honor of Steven L. Kleiman},
               journal={Comm. Algebra},
               volume={31},
               date={2003},
               number={8},
            }
	    \bib{ArsVis}{article}{			
			author={Arsie, A.},
			author={Vistoli, A.},			
			title={Stacks of cyclic covers of projective spaces},			
			journal={Compos. Math.},
			volume={140},			
			date={2004},			
			number={3}	}
		\bib{Bea}{article}{
           author={Beauville, A.},
           title={Prym varieties and the Schottky problem},
           journal={Invent. Math.},
           volume={41},
           date={1977},
           number={2},
        }	

		\bib{Cor}{article}{
               author={Cornalba, M.},
               title={The Picard group of the moduli stack of stable hyperelliptic
               curves},
               journal={Atti Accad. Naz. Lincei Rend. Lincei Mat. Appl.},
               volume={18},
               date={2007},
               number={1},
            }
        \bib{CH}{article}{
			author={Cornalba, M.},
			author={Harris, J.},
			title={Divisor classes associated to families of stable varieties, with
				applications to the moduli space of curves},
			journal={Ann. Sci. \'{E}cole Norm. Sup. (4)},
			volume={21},
			date={1988},
			number={3},
		}    

		\bib{DilPir}{article}{
		author={Di Lorenzo, A.},
		author={Pirisi, R.},
		title={A complete description of the cohomological invariants of even genus Hyperelliptic curves},
		journal={ar{X}iv:1911.04005 [math.AG]}
		}
		\bib{DilPir2}{article}{
		author={Di Lorenzo, A.},
		author={Pirisi, R.},
		title={Brauer groups of moduli of hyperelliptic curves via cohomological invariants},
		journal={ar{X}iv:2002.11065 [math.AG]}
		}

		\bib{EG}{article}{
			author={Edidin, D.},
			author={Graham, W.},
			title={Equivariant intersection theory},
			journal={Invent. Math.},
			volume={131},
			date={1998},
			number={3},
		}


		\bib{GMS}{collection}{
			author={Garibaldi, S.},
			author={Merkurjev, A.},
			author={Serre, J.-P.},
			title={Cohomological invariants in Galois cohomology},
			series={University Lecture Series},
			volume={28},
			publisher={American Mathematical Society, Providence, RI},
			date={2003},
		}
		\bib{GilHir}{article}{
		author={Gille, S.},
		author={Hirsch, C.},
		title={On the splitting principle for cohomological invariants of reflection groups},
		journal={ar{X}iv:1908.08146 [math.AG]}
		}
		\bib{GV}{article}{
		author={Gorchinskiy, S.},
		author={Viviani, F.},
		title={Picard group of moduli of hyperelliptic curves}
		journal={Math. Z.},
		volume={258},
		pages={319–331},
		date={2008},
		}
		
        \bib{Guil}{article}{			
			author={Guillot, P.},
			title={Geometric methods for cohomological invariants},			
			journal={Doc. Math.},
			volume={12},			
			date={2007},
		}
		\bib{Har}{article}{
		    author={Harper, A.},
		    title={Factorization for stacks and boundary complexes},
		    eprint={https://arxiv.org/abs/1706.07999v1},
		    }
        \bib{HM}{article}{
           author={Harris, J.},
           author={Mumford, D.},
           title={On the Kodaira dimension of the moduli space of curves},
           note={With an appendix by William Fulton},
           journal={Invent. Math.},
           volume={67},
           date={1982},
        }
        \bib{KreProj}{article}{
        author={Kresch, A.},
        title={On the geometry of Deligne-Mumford stacks},
        series={Proc. Sympos. Pure Math.},
        volume={80},
        publisher={Amer. Math. Soc.},
        date={2009},
        pages={259–271},
        }
		\bib{PirAlgStack}{article}{			
			author={Pirisi, R.},
			title={Cohomological invariants of algebraic stacks},			
			journal={Trans. Amer. Math. Soc.},
			volume={370},			
			date={2018},			
			number={3}	}
		\bib{Rost}{article}{			
			author={Rost, M.},
			title={Chow groups with coefficients},			
			journal={Doc. Math.},
			volume={1},			
			date={1996},			
			number={16}	}
		\bib{Sca}{thesis}{
		    author={Scavia, F.},
		    title={The stack of admissible double covers},
		    type={{MS}c thesis},
		    organization={Scuola Normale Superiore},
		    date={2017}
		    }
		\bib{StPr}{misc}{
		    label={Stacks},
		    title={{S}tacks-{P}roject},
            author={The {Stacks Project Authors}},
            publisher = {https://stacks.math.columbia.edu},
            date = {2018},
        }	
        \bib{Voe}{article}{
		    author={Voevodsky, V.},
		    title={On motivic cohomology with $Z/l$-coefficients},
		    journal={Annals of Mathematics},
		    volume={174},
		    pages={401-438},
		    date={2011},
		}
		\bib{Yam}{article}{
		    author={Yamaki, K.},
		    title={{C}ornalba-{H}arris equality for semistable hyperelliptic curves in positive characteristic},
		    journal={Asian J. Math.},
		    volume={8},
		    number={3},
		    pages={409-426},
		    date={2004},
		}
		\bib{Wit}{article}{
			author={Witt, E.},
			title={Theorie der quadratischen Formen in beliebigen K\"orpern},
			language={German},
			journal={J. Reine Angew. Math.},
			volume={176},
			date={1937}
		}
		
		
	\end{biblist}
\end{bibdiv}
\end{document}